\documentclass[11pt,reqno]{amsart}

\usepackage{amscd,amsmath,amssymb,latexsym,amsthm,amsfonts}
\usepackage{enumerate,color,graphicx}
\usepackage[all]{xypic}
\usepackage{tikz}
\renewcommand\thesubsection{\thesection.\arabic{subsection}}
\renewcommand\subsection{\refstepcounter{subsection}\par\addvspace{\bigskipamount}\noindent{\normalfont\normalsize\bfseries\thesubsection.\hspace*{.5em}}}

\usepackage{hyperref}
\hypersetup{
 unicode=false, % non-Latin characters in Acrobat bookmarks
 pdftoolbar=true, % show Acrobat toolbar?
 pdfmenubar=true, % show Acrobat menu?
 pdffitwindow=false, % window fit to page when opened
 pdfstartview={FitH}, % fits the width of the page to the window
 colorlinks=true, % false: boxed links; true: colored links
 linkcolor=magenta, % color of internal links %% blue
 citecolor=blue, % color of links to bibliography %% blue
 filecolor=black, % color of file links
 urlcolor=black % color of external links % % blue
}

\theoremstyle{theorem}
\newtheorem{theo}{Theorem}[section]
\newtheorem{coro}[theo]{Corollary}
\newtheorem{lemma}[theo]{Lemma}
\newtheorem{prop}[theo]{Proposition}

\newtheorem{theosub}{Theorem}[subsection]
\newtheorem{propsub}[theosub]{Proposition}

\theoremstyle{definition}
\newtheorem{note}[theo]{Note}
\newtheorem{rema}[theo]{Remark}
\newtheorem{defin}[theo]{Definition}
\newtheorem{exams}[theo]{Examples}

\def\nid{\noindent}
\def\cf{\mathfrak{c}}
\def\of{\mathfrak{o}}

\def\sue{\subseteq}
\def\OS{\text{\sf O}}

\def\SL{\mathsf{S}}

\def\ur{{{\rlap{$\ $}\hbox{$\uparrow$}}}}%
\def\setof#1#2{\{#1 \ | \ #2\} }

\def\setof#1#2{\{#1 \, |\, #2\} }
\def\qtq#1{\quad\text{#1}\quad}

\def\TOP{\text{\bf Top}}

\def\loc{\text{\bf Loc}}
\newcommand{\tbigcap}{\mathop{\textstyle \bigcap}}
\newcommand{\tbigcup}{\mathop{\textstyle \bigcup }}%%
\newcommand{\tbigvee}{\mathop{\textstyle \bigvee }}%%
\newcommand{\tbigwedge}{\mathop{\textstyle \bigwedge }}%%

\def \lf{f_{\leftarrow}}
\def \rf{f_{\rightarrow}}
\def \pf {f^{-1}}
\def \pg {g^{-1}}
\def \clf {f^\cf _\leftarrow}
\def \clg {g^\cf _\leftarrow}

\def \mt {\mathcal{T}}
\def \ms {\mathcal{S}}
\def \mo {\mathcal{O}}
\def \mc {\mathcal{C}}
\def \mcs {\mathcal{C}(\mathcal{S})}
\def \mos {\mathcal{O}(\mathcal{S})}
\def \mot {\mathcal{O}(\mathcal{T})}

\def \bds {\mathcal{B}d(\mathcal{S})}
\def \bdt {\mathcal{B}d(\mathcal{T})}
\def\Bds{\text{\bf B0ds}}

\begin{document}

\title[Basic zero-dimensional spaces: a unifying framework ]{Basic zero-dimensional spaces: a unifying framework for continuity and openness}

\author[João Areias]{João Areias}

\address{\hspace*{-\parindent}Department of Mathematics, University of Coimbra,  3000-143 Coimbra,\newline Portugal \newline {\it Email address}: {\tt uc2020217955@student.uc.pt}}

\author[Jorge Picado]{Jorge Picado}

\address{\hspace*{-\parindent}CMUC, Department of Mathematics, University of Coimbra,\newline  3000-143 Coimbra, Portugal \newline {\it Email address}: {\tt picado@mat.uc.pt}}



\date{\today}

\begin{abstract}
Following a suggestion in \cite{EPP}, we establish the category of basic zero-dimensional spaces and their basic continuous maps. The objects are closure spaces with a distributive closure system
containing a specified meet-base of complemented members that make the category a common generalization of those of closure spaces, zero-dimensional topological spaces, Heyting semilattices, and locales (frames). We treat images, preimages, continuity, openness and closedness of maps between basic zero-dimensional spaces. Among our main results, we have a useful description of preimages for basic continuous maps (analogous to the one for localic preimage maps as coframe homomorphisms),
 and a Joyal-Tierney type theorem for basic open maps. A novel aspect of the category of basic zero-dimensional spaces is a duality principle that, among other things, enables results for basic closed maps to be obtained for free from those for basic open maps.
\end{abstract}

\maketitle

\section{Introduction}

%Articles such as [14], [10], and [7] have demonstrated that many results in point-free topology can be viewed as instances of more general results in the non-complete setting of Heyting semilattices. Working in this broader, non-complete framework not only sheds new light on existing results and their proofs, but also shows that the existence of arbitrary joins and meets is often unnecessary.
%
%The present article continues this line of research. Its primary motivation is a suggestion made at the end of [10] to investigate a category of so-called basic zero-dimensional (closure) spaces and basic continuous maps. Similar to the categories considered in [4] and [5], this category is intended to provide a unified framework encompassing the various categories of spaces discussed there, including closure spaces, zero-dimensional topological spaces, Heyting semilattices, and, in particular, locales, together with their corresponding notions of continuous maps.

Articles such as \cite{PPT}, \cite{EPP}, and \cite{ME20} have demonstrated that several results in point-free topology can be viewed as instances of more general results in the non-complete setting of Heyting semilattices.
Working in this broader, non-complete framework not only sheds new light on existing results and their proofs, but also shows that the existence of infinite joins and meets is often unnecessary.

The present article continues this line of research. Its primary motivation is a suggestion made at the end of \cite{EPP} to investigate a category of so-called {\em basic zero-dimensional \textup(closure\textup) spaces}
and {\em basic continuous maps}, similar to categories considered in
\cite{ME84} and \cite{ME04}. This category is intended to provide a  unified framework encompassing  such categories as those of closure spaces, zero-dimensional topological spaces, Heyting semilattices, and, in particular, locales, together with their corresponding notions of continuous maps.

\bigskip
The motivation for introducing this category of spaces is provided by the following two guiding examples. Let  $L$ be a locale (that is, a complete Heyting algebra). The lattice $\SL(L)$ of all subobjects ({\em sublocales}) is a coframe. Meets in $\SL(L)$ are just intersections; hence $\SL(L)$ is a closure system in $L$.

For each $a\in L$, there are two basic elements of $\SL(L)$, the {\em open} sublocale $\of a=\setof {a\to x}{x\in H}$  and the {\em closed} sublocale $\cf a=\ur a$, complemented to each other.
A fundamental property of the theory of locales asserts that every $S\in\SL(L)$ is of the form
\[
S=\tbigcap_{i\in J}(\cf a_i\vee \of b_i)=\tbigcap_{i\in J}(\cf a_i\vee \neg \cf b_i)
\]
(this expresses the fact that the  duals of lattices of
sublocales are {\em zero-dimensional}).

More generally, let $H$ be a Heyting semilattice with its (complete) lattice of ideals $\mathcal{I}dl(H)$. We still have the open and closed $\of a$, $\cf a$, which are elements of $\mathcal{I}dl(H)$, complemented to each other. The joins $\cf a\vee \of b$ generate in $\mathcal{I}dl(H)$ the class $\mathcal{B}d(H)$ of {\em basic domains} of $H$, that is, the
subsets of $H$ of the form
$$\tbigcap_{i\in J}(\cf a_i^1\vee \of b_i^1\vee \dots \vee \cf a_i^{n_i}\vee \of b_i^{n_i}).$$
Again, this is a coframe and, in particular, a closure system in $H$.

It is easy to see that basic domains are a generalization of sublocales.

The triples
$$(L,\cf L,\SL(L))  \quad\mbox{ and }\quad (H, \cf H, \mathcal{B}d(H))$$ are, among others, the motivation for {\em basic zero-dimensional spaces} (briefly, b0-space). These are
 triples $\mathcal{S}=(X,\mathcal{C}, \mathcal{B}d)$, where
	\begin{itemize}
		\item[1.]  $X$ is a set,
		
		\item[2.] $\mathcal{B}d$ is a closure system on $X$, distributive as a lattice, and
		
		\item[3.]  $\mathcal{C}\subseteq \mathcal{B}d$ is such that every $C\in \mc$  has a complement $\neg C$ in $\mathcal{B}d$, and every $M\in\mathcal{B}d$ is of the form
    \[M=\tbigcap_{i\in J}(C_i^1\vee \neg D_i^1\vee \dots \vee C_i^{n_i}\vee \neg D_i^{n_i}),\]
     for some $C^1_i, D^1_i, \ldots, C_i^{n_i}, D_i^{n_i}\in \mathcal{C}$.
	\end{itemize}

\bigskip
We will show in this paper that working within the context of b0-spaces not only makes the results far more general, but also illuminates the phenomena behind the proofs better.
Our main goal is to confirm that
b0-spaces provide a useful unifying framework for studying continuity and openness of maps, unifying
the modelling of continuous maps and open continuous maps in point-free topology, point-set topology and elsewhere \cite{EPP}.

In \cite{EPP}, the authors showed that some results in locales are consequences of some more general results in Heyting semilattices. In doing so they identified two main constraints of Heyting semilattices when one takes \emph{nuclear ranges} (see \ref{sec nuc e id} below) as the generalization of sublocales:
the system of nuclear ranges is not a coframe in general (in fact, it may even have no meets), and there is no non-complete analogue for the preimages of localic maps.

 By applying our results on b0-spaces to the category of Heyting semilattices, we can immediately overcome these problems: the system of substructures of a b0-space is always a coframe, and an analogue of localic preimages
exists for basic continuous maps.

One of the main results in \cite{EPP} is the generalization to Heyting semilattices of the Joyal-Tierney characterization of open localic maps.
In this paper, we obtain a similar result for basic open maps between b0-spaces. By applying it to Heyting semilattices we get an alternative proof for the Joyal-Tierney type theorem in Heyting semilattices \cite{EPP}.
Another notable feature of the category of b0-spaces is a duality principle that, among other things, enables results for basic closed maps to be obtained for free from those for basic open maps.

\bigskip
The paper is organized as follows:

Section 2 presents the necessary background information and terminology. Section 3 introduces the category of b0-spaces and basic continuous maps. Section 4 introduces basic images and preimages. We then characterise basic continuous maps, concluding that a map admits a localic-type preimage if and only if it is basic continuous (Section 5). We also prove that the basic preimage is a coframe homomorphism (Section 6). Section 7 applies the theory of b0-spaces to Heyting semilattices.
Section 8 characterises basic open maps, and Section 9 applies this to Heyting semilattices, providing a new proof of the Joyal–Tierney theorem for Heyting semilattices.
In Section 10, we present a duality principle in the category of b0-spaces, and use it to characterize basic closed maps and to present in the final section (Section 11) a generalization, to the non-complete setting of Heyting lattices, of the well-known theorem that characterizes closed localic maps.

Throughout the paper, we illustrate our main results by formulating them in the context of Heyting (semi)lattices.

\section{Preliminaries}\label{prelimi}

In this section we collect some facts and notation about closure operators, Heyting semilattices, frames and locales, coframes, and Galois adjunctions that we will need in
the following.
We use the standard notions and notation for partially
ordered sets and lattices as e.g. in \cite{DP} and \cite{PP}.

\subsection{\bf Galois adjunctions.}\setcounter{subsection}{1}
\setcounter{theosub}{0}
Given two partially ordered sets $X$ and $Y$, a {\em Galois adjunction} \cite{EKMS} between them consists of a pair of order-preserving maps $f \colon  X \to Y$ and $g \colon  Y \to X$ such that
$$
f(x) \leq y \Leftrightarrow x \leq g(y).
$$
  One says that $f$ is {\em left adjoint} to $g$ (and $g$ is {\em right adjoint} to $f$) and writes $f\dashv g$. Left and right adjoints, when they exist, are unique.
It is standard that (cf. \cite{EKMS,PP})
\begin{enumerate}[(G1)]
	\item If $f\dashv g$, then $fg\le \mathrm{id}_Y$ and $\mathrm{id}_X\le gf$,
	\item left adjoints preserve all existing suprema and right adjoints preserve all existing infima, and
	\item
	if $X,Y$ are complete lattices, then each $f\colon  X \to Y$ preserving all suprema has a right adjoint (given by  $g(y)=\tbigvee\{x\in X\mid f(x)\le y\}$), and each $g \colon  Y \to X$ preserving all infima has a left adjoint (given by  $f(x)=\tbigwedge\{y\in Y\mid g(y)\ge x\}$).
\end{enumerate}

\subsection{\bf Closure operators and closure systems.}
\setcounter{subsection}{2}
\setcounter{theosub}{0}
	A \emph{closure operator} on a poset $P$ is a map $\gamma \colon P\to P$ satisfying
\begin{equation}
\label{propriedade do fecho}\tag{2.2.1}
	a\leq \gamma(b)\Leftrightarrow \gamma(a)\leq \gamma(b) \quad\mbox{for every } a,b\in P.
\end{equation}

Given a set $X$, one usually talks about a closure operator on $X$ to mean a closure operator on $\mathcal{P}(X)$.
Closure operators on $X$ are closely related to the notion of a closure system on $X$.
 A {\em closure system} on a set $X$ is a subset $\mathcal C$ of $\mathcal{P}(X)$ such that $X\in \mathcal{C}$ and $\tbigcap\mathcal{X}\in \mathcal{C}$ for every $\mathcal{X} \subseteq \mathcal{C}$. A closure system is said to be \emph{topological} if it is closed under finite unions. One has the obvious

\begin{propsub}\label{sistemas e operadores de fecho} For each
closure system $\mathcal{C}$ on $X$,  the operator
	$\Gamma_\mathcal{C}(A)=\tbigcap\{Y\in \mathcal{C}\mid A\subseteq Y \}$
    is a closure operator.
\end{propsub}

There is a canonical way \cite{ME09} of constructing a closure system from an $\mathcal{X}\subseteq\mathcal{P}(X)$.
Given a set $X$ and $\mathcal{X}\subseteq\mathcal{P}(X)$, set
\begin{align*}
	\Omega\mathcal{X}&= \{\tbigcap\mathcal{Y}\mid \mathcal{Y}\subseteq\mathcal{X}\},\quad\quad\quad\quad\quad\quad
	\mho \mathcal{X}=\{\tbigcup\mathcal{Y}\mid \mathcal{Y}\subseteq\mathcal{X}\},\\
	\Omega^\mathfrak{f} \mathcal{X}&= \{\tbigcap\mathcal{Y}\mid \mathcal{Y}\subseteq\mathcal{X} \text{,  $\mathcal{Y}$ is finite}\},
	\quad \mho^\mathfrak{f} \mathcal{X}= \{\tbigcup\mathcal{Y}\mid \mathcal{Y}\subseteq\mathcal{X}\text{,  $\mathcal{Y}$ is finite}\}.
\end{align*}

\begin{propsub}[\cite{ME09}]\label{prop1.4.6}
	Among the closure systems on $X$ containing $\mathcal{X}$,
	\begin{enumerate}[\em(1)]
		\item $\Omega\mathcal{X}$ is the least closure system, and
		\item $\Omega \mho^\mathfrak{f}\mathcal{X}$ is the least topological closure system.
	\end{enumerate}
\end{propsub}

	A \emph{closure space} is a pair $(X, \mathcal{C})$, where $X$ is a set and $\mathcal{C}$ is a closure system on $X$.
	If $(X,\mathcal{C}_X)$ and $(Y, \mathcal{C}_Y)$ are closure spaces, a map $f\colon X\to Y$ is said to be {\em continuous} if $\pf[C]\in \mathcal{\mathcal{C}}_X$ for every $C\in \mathcal{C}_Y$.
We refer to \cite{ME09} for more information about closure operators and closure spaces.

\subsection{\bf Heyting semilattices.}\label{Heytsemil}
\setcounter{theosub}{0}
A \emph{Heyting semilattice} $H$ (also known as {\em implicative semilattice} \cite{ME20}, or {\em Brouwerian semilattice} \cite{Ko1}) is a meet-semilattice (with top element 1) in which every unary operator $\lambda_a=a\wedge(-)$ has a right adjoint $\gamma_a=a\to(-)$, that is,
\begin{equation}\label{Heyting}\tag{2.3.1}
a\wedge b \leq c\Leftrightarrow b\leq a\to c\quad\mbox{for every }a, b, c\in H.
\end{equation}
Lattices with an (Heyting) operation $\to$ satisfying \eqref{Heyting} are called
 \emph{Heyting lattices} (also {\em Heyting algebras}).

 We list below the basic properties of the Heyting operator in Heyting semilattices that we will need (\cite{Ko1}, Section 1]:
 	\begin{enumerate}[(H1)]
		\item $(a\to b) \wedge (a\to  c)=a\to ( b\wedge c)$.
		\item $a\leq b $ if and only if $a\to b=1$.
		\item  $b\leq a\to b$.
		\item $a\to(b\to c)=(a\wedge b)\to c =b\to (a\to c)$.
		\item $a\leq (a\to b)\to b$.
		\item  $b=((a\to b)\to b)\wedge (a\to b)$.
	\end{enumerate}

In a Heyting semilattice $H$ one has for each $a\in H$ the \emph{open set}
$$\of a=\setof{x\in H}{a\to x=x}=\setof {a\to x}{x\in H}$$ and the \emph{closed set} $$\cf a=\ur a=\{x\in H\mid x\ge a\}\ \mbox{ (\cite{EPP})}.$$ The system of all open subsets of $H$ is denoted by $\of H$, and the system of all closed subsets by $\cf H$. The subset $\{1\}$ of $H$ is denoted by $\OS$.

For every $a,b \in H$,
\begin{equation}\tag{2.3.2}
	a\leq b\Leftrightarrow \cf b \sue \cf a \qtq{and} a\leq b \Leftrightarrow\of a\sue \of b.
\end{equation}

Given $M,K\subseteq H,$ we will also need to consider the set
$$M\vartriangle K=\{m\wedge k \mid m\in M, \: k\in K\}.$$
When $H$ is a Heyting lattice, we have:
\[
\begin{aligned}
	&\of(0)=\OS,\; \of(1)=H,\ \ \of a\cap\of b=\of(a\wedge b) \ \text{ and }\ \of a \vartriangle \of b=\of(a \vee b),\\
	&\cf(1)=\OS,\; \cf(0)=H,\ \ \cf a\vartriangle\cf b=\cf(a\wedge b) \ \text{ and }\ \cf a\cap\cf b=\cf(a \vee b).
\end{aligned}
\]

A map between Heyting semilattices is \emph{continuous} (also known as an \emph{$l$-morphism} \cite{EPP}) if it has a left adjoint that preserves finite meets (including the top element). Continuous maps between Heyting semilattices are characterized by the following result:

\begin{propsub}[\cite{EPP}]\label{prop carac cont}
	A map between Heyting semilattices $f\colon H_1\to H_2$ is a continuous map if and only if the following conditions hold:
	\begin{enumerate}[\em(1)]
		\item $f(a)=1\ \Rightarrow \ a=1$.
		\item $f$ has a left adjoint $f^*$.
		\item For every $a\in H_1$ and $b\in H_2$, $f(f^*(b)\to a)=b\to f(a)$.
	\end{enumerate}
\end{propsub}

Condition (3) is usually referred to as the {\em Frobenius identity} (or the {\em $\to$-semilinearity} of $f$ \cite{EPP}).

\subsection{\bf Ideals and nuclear ranges.}\label{sec nuc e id}
\setcounter{theosub}{0}
\setcounter{subsection}{4} Let $H$ be a Heyting semilattice.
	A sub-semilattice $T\subseteq H$ is an {\em ideal}  of $H$ (also known as a {\em total sub-algebra} \cite{Ko1}) if $x\to t \in T$ for every $x\in H$ and  $t\in T$.
	The system of all ideals of $H$ will be denoted by $$\mathcal{I}dl(H).$$
This is a closure system, hence a complete lattice. Moreover:

\begin{propsub}[\cite{Ko1}]
     $\mathcal{I}dl(H)$ is a distributive lattice with
    $T_1\wedge T_2=T_1\vartriangle T_2.$
\end{propsub}

\begin{propsub}[\cite{PPT}]
    For each $a\in H$, $\cf a$ and $\of a$ belong to $\mathcal{I}dl(H)$ and are complemented to each other.
\end{propsub}

 A subset $N\subseteq H$ is a {\em nuclear range} (also known as a \emph{strong ideal} \cite{PPT}) if it is the image of some nucleus on $H$ (recall that a \emph{nucleus} on a meet-semilattice is a closure operator that preserves binary meets).
The poset of all nuclear ranges on $H$ ordered by inclusion will be denoted by $$\mathcal{N}uc(H).$$

\begin{propsub}[\cite{ME20}]\label{abertos sao nuc}
	For each $a\in H$, $\of a \in \mathcal{N}uc(H)$. Moreover,  $\cf a \in \mathcal{N}uc(H)$ for every $a\in H$ if and only if $H$ is a Heyting lattice.
\end{propsub}

Nuclear ranges behave well under images of continuous maps:

\begin{propsub}[\cite{EPP}]\label{cont preservam nuc}
The image of a nuclear range under a continuous map is a nuclear range.
\end{propsub}

\subsection{\bf Frames and locales.}
\setcounter{subsection}{5}
\setcounter{theosub}{0}
A \emph{frame}  (also known as a {\em locale}) is a complete lattice $L$ satisfying the infinite distributivity law
$$a\wedge \tbigvee B=\tbigvee \{a\wedge b\mid b\in B\}$$
for all $a\in L$ and $B\subseteq L$. Frames are precisely complete Heyting lattices, with the Heyting operation given by $a\to b=\tbigvee\{x\mid x\wedge a\le b\}$.
Standard examples of frames are the lattices $\mathcal O X$ of open sets of topological spaces.
A {\em coframe} $L$ is the dual of a frame, that is, a complete lattice satisfying
$a\vee \tbigwedge B=\tbigwedge \{a\vee b\mid b\in B\} $
for all $a\in L$ and $B\subseteq L$.

A \emph{frame homomorphism} $f\colon L\to M$ is a map between frames that preserves arbitrary joins (including the bottom element 0) and finite meets (including the top element 1). Similarly, a {\em coframe homomorphism} preserves arbitrary meets and finite joins.
\emph{Localic maps} are the right adjoints of frame homomorphisms. They are precisely the continuous maps between complete Heyting algebras and are thus characterized by conditions (1)-(3) in Proposition \ref{prop carac cont}.

A {\em sublocale} of a locale $L$ is a subset $S\sue L$ such that
	\begin{enumerate}[(S1)]
		\item
		$S$ is a {\em meet-subset}, that is, for every $M \subseteq S$, $\tbigwedge M\in S$.
		\item
		For every $s\in S$ and every $x\in L$, $x \to s\in S$.
	\end{enumerate}
The open and closed sets $\of a$ and $\cf a$ are examples of sublocales of $L$. The least sublocale of $L$ is the set $\OS=\{1\}$.

The system $\SL(L)$ of all sublocales of $L$ is a coframe  with a fairly transparent structure \cite[pp. 28-29]{PP}:
\begin{equation}
\tag{2.5.1}\label{joinsub}
	\tbigwedge_{i\in J} S_i
	= \tbigcap_{i\in J} S_i
	\qtq{and} \tbigvee_{i\in J} S_{i}
	= \setof{\tbigwedge M}{M\subseteq \tbigcup_{i\in J} S_i}.
\end{equation}

Let $f\colon L\to M$ be a localic map (with left adjoint $f^*$) and $S\sue L$ and $T\sue M$ sublocales.
Then the standard set theoretic image $f[S]$ is easily seen to be a sublocale of $M$. The standard preimage $f^{-1}[T]$ is generally not a sublocale, but it is closed under meets and hence, by the formula for suprema in \eqref{joinsub} there exists a largest sublocale contained in $f^{-1}[T]$ which we denote by $f_{-1}[T]$. This is the localic preimage
 of $T$. There is the obvious Galois adjunction
\[f[S]\sue T\ \mbox{ iff }\ S\sue f_{-1}[T].\]
Consequently,  $f_{-1}[-]$ preserves intersections. Moreover,
\[f_{-1}[\OS]=\OS,\ \ f_{-1}[\cf(a)]=\cf(f^*(a))\ \mbox{ and }\ f_{-1}[\of(a)]=\of(f^*(a)).\]

For more about frames, locales and point-free topology consult \cite{PP}.

\subsection{\bf Constructing coframes inside complete lattices.}
\setcounter{subsection}{6}
\setcounter{theosub}{0}
Given a complete lattice $L$ and  $M\subseteq L$, set
\begin{align*}
	\Lambda M &=\{\tbigwedge K\mid K\sue M\},  &&\rotatebox[origin=c]{180}{$\Lambda$}{ M=\{\tbigvee K\mid K\sue M\}},\\
	\Lambda^\mathfrak{f} M&=\{\tbigwedge K\mid K\sue M \text{, $K$ is finite}\},  &&\rotatebox[origin=c]{180}{$\Lambda$}^\mathfrak{f} M=\{\tbigvee K\mid K\sue M\text{, $K$ is finite}\}.
\end{align*}

%Whenever we need to deal with more than one complete lattice, we will put the name of the lattice in subscript: $\Lambda_L(-)$, $\rotatebox[origin=c]{180}{$\Lambda$}_L(-)$,  $\Lambda^\mathfrak{f}_L(-) $ and $\rotatebox[origin=c]{180}{$\Lambda$}_L^\mathfrak{f}(-)$.

The following important result of M. Erné \cite{ME} will be needed in the sequel.  For the sake of completeness, we include a proof, as \cite{ME} is not easily available.
%convenience of the reader.

\begin{propsub}\label{prop tecnica}
	Let $L$ be a distributive complete lattice, and let $M$ be a set of complemented elements of $L$. Then $T=\Lambda \rotatebox[origin=c]{180}{$\Lambda$}^\mathfrak{f} M$ is a coframe whose binary joins are given by the formula
	\begin{equation}\tag{2.6.1}\label{152ast}
		(\tbigwedge_i a_i) \stackrel{T}{\vee} (\tbigwedge_j b_j)=\tbigwedge_{i,j}(a_i  \vee b_j) \quad \mbox{   for every $a_i, b_j \in \rotatebox[origin=c]{180}{$\Lambda$}^\mathfrak{f} M$.}
	\end{equation}
\end{propsub}

\begin{proof}
We can assume that $M=\rotatebox[origin=c]{180}{$\Lambda$}^\mathfrak{f} M$, (otherwise, just consider $M'=\rotatebox[origin=c]{180}{$\Lambda$}^\mathfrak{f} M$, and observe that $\Lambda \rotatebox[origin=c]{180}{$\Lambda$}^\mathfrak{f} M'= \Lambda \rotatebox[origin=c]{180}{$\Lambda$}^\mathfrak{f} M= \Lambda M'$).
	
	It is obvious that $T$ is complete. To show that the joins are given by \eqref{152ast} let $a, b \in T$ and $a_i, b_j \in M$ such that $a=\tbigwedge_i a_i$ and $b=\tbigwedge_j b_j$. It is clear that $$a\stackrel{T}{\vee} b \leq \tbigwedge_{i,j}(a_i \vee b_j).$$
	For the converse inequality, consider $m\in M$ such that $m\geq a, b$. In a distributive complete lattice, every complemented element $m$ satisfies the distributive law $m\vee \tbigwedge_i x_i=\tbigwedge_i(m\vee x_i)$ for any system $\{x_i\}_i$ of elements of $L$
(see e.g. \cite[VI.4.4.1]{PP}). Therefore
	$$m=m\vee a= m \vee \tbigwedge_i a_i= \tbigwedge_i (m \vee a_i)$$
	and
	$$m= \tbigwedge_j (b_j \vee m)= \tbigwedge_j (b_j \vee \tbigwedge_i (m \vee a_i))= \tbigwedge_{i,j} (b_j \vee a_i \vee m )\geq \tbigwedge_{i,j}(a_i \vee b_j).$$
	Since this holds for every $m\in M$ and $M$ is a meet-base of $T$, it follows that $\tbigwedge_{i,j}(a_i \vee b_j)\leq a \stackrel{T}{\vee} b$.
	
	Finally, to see that $T$ is a coframe, take $c_k \in T$ for $ k \in K$ and let $$Y=\{m \in M \mid c_k \leq m\text{ for some } k\in K\}.$$ Clearly $\tbigwedge_k c_k = \tbigwedge Y$ and, moreover,
	\begin{align*}a \stackrel{T}{\vee} \tbigwedge_{k\in K} c_k &= (\tbigwedge_{i\in J} a_i) \stackrel{T}{\vee}  \tbigwedge_{m\in Y} m = \tbigwedge_{m\in Y,i\in J} (m \vee a_i ) \\ &\geq \tbigwedge_{i\in J,k\in K}(a_i \vee c_k)\geq \tbigwedge_{k\in K} (a\vee c_k).\qedhere
\end{align*}
\end{proof}

\section{The category of basic zero-dimensional spaces}\label{sec b0}

We will now introduce our basic setting, the category of basic zero-dimensional spaces and basic continuous maps. These notions were suggested in \cite[\S 7]{EPP}. However, a few adjustments need to be made to ensure things work well.

\begin{defin}\label{definiçao bods}
	A \emph{basic zero-dimensional} space (briefly, {\em b0-space}) is  a triple $\mathcal{S}=(X,\mathcal{C}, \mathcal{B}d)$, where
	\begin{itemize}
		\item[--]  $X$ is a set,
		
		\item[--] $\mathcal{B}d$ is a closure system on $X$, distributive as a lattice, and
		
		\item[--]  $\mathcal{C}\subseteq \mathcal{B}d$ is such that every $C\in \mc$  has a complement $\neg C$ in $\mathcal{B}d$, and every $M\in\mathcal{B}d$ is of the form
    \[M=\tbigcap_{i\in J}(C_i^1\vee \neg D_i^1\vee \dots \vee C_i^{n_i}\vee \neg D_i^{n_i}),\]
     for some $C^1_i, D^1_i, \ldots, C_i^{n_i}, D_i^{n_i}\in \mathcal{C}$.
	\end{itemize}

The bottom element of $\mathcal{B}d$ is denoted by $\OS$ and the set $\{\neg C \mid C\in \mathcal{C}\}$ by $\mathcal{O}$.
To avoid possible notation conflicts, we sometimes write
$$X=|\mathcal{S}|,\ \mathcal{C}=\mathcal{C}(\mathcal{S}),\ \mathcal{O}=\mathcal{O}(\mathcal{S}), \ \mbox{ and }\ \mathcal{B}d=\mathcal{B}d(\mathcal{S}).$$
If we set  $\mathcal{B}=\{B \vee \neg C \mid B,\, C \in \mathcal{C}\}$, we have
$\mathcal{B}d=\Omega\rotatebox[origin=c]{180}{$\Lambda$}^\mathfrak{f}_{\mathcal{B}d}(\mathcal{B}).$

The elements of $\mathcal{B}d$ will be called \emph{basic domains} and $\mathcal{B}$ will be called a {\em basis} of $\mathcal{B}d$.
Although $\mathcal{C}$ does not need to be a closure system, we will refer to its elements as \emph{closed \textup(domains\textup)} and to their complements in $\mathcal{B}d$ as \emph{open \textup(domains\textup)}.
\end{defin}

As an immediate application of Proposition \ref{prop tecnica}, we have:
\begin{prop}
	If $\mathcal{S}=(X,\mathcal{C}, \mathcal{B}d)$ is a b0-space, then $\mathcal{B}d$ is a coframe.
\end{prop}

\bigskip
One of the motivating examples of $b0$-spaces is the triple $$(L,\cf L,\SL (L))$$ for any locale $L$,
as every sublocale $S$ is of the form
$$S=\tbigcap \{\cf a \vee \of b\mid S\sue \cf a \vee \of b\}$$
(this result, proved in \cite{PP}, expresses the well-known fact that coframes of sublocales are duals of zero-dimensional frames).

\begin{defin}
	Let $\mathcal{S}$ and $\mathcal{T}$ be b0-spaces. A {\em basic continuous map} $f\colon \mathcal{S}\to \mathcal{T}$ is a
	map $f\colon |\mathcal{S}|\to |\mathcal{T}|$ with the following properties:
	\begin{enumerate}[(BC1)]
		\item
		$f^{-1}[\OS]=\OS$.
		\item There exists a map $f_{\leftarrow}^\cf\colon \mathcal{C}(\mathcal{T})\to\mathcal{C}(\mathcal{S})$ such that
		$$f[B]\subseteq C \mbox{ iff } B \subseteq f_{\leftarrow}^\cf [C]\quad \mbox{ for every }B\in \mathcal{C}(\mathcal{S})\mbox{ and }C\in \mathcal{C}(\mathcal{T}).$$
		\item For every $C\in \mathcal{C}(\mathcal{T})$,
		$\neg f_{\leftarrow}^\cf[C]\subseteq f^{-1}[\neg C].$
		\item For every $C_i,\;D_i \in \mathcal{C}(\mathcal{T})$ ($i=1,2,\ldots,n$),
		$$\tbigvee_{i=1}^n \left (f_{\leftarrow}^\cf[C_i]\vee \neg f_{\leftarrow}^\cf[D_i] \right )\subseteq f^{-1}[\tbigvee_{i=1}^n (C_i \vee \neg D_i)].$$
	\end{enumerate}
	The map $\clf$  is called the \emph{$\cf$-basic preimage of $f$}. We will refer to the pair $(f,\clf )$ as a {\em $\cf$-basic continuous pair}, to indicate that $f$ is a basic continuous map and that $\clf$ is the associated $\cf$-basic preimage map.
%\footnote{All the (BC) conditions above are independent from each other (for details see \cite{tese}).}
\end{defin}

\subsubsection*{\bf Example.} Any localic map
$f\colon L\to M$ induces a basic continuous map $$f\colon(L, \cf L, \SL(L))\to(M, \cf M, \SL(M))$$ with $\clf$ given by $\clf[\cf a]=f_{-1}[\cf a]$.

\bigskip
It is obvious that $\clf[C]$ is the largest closed domain contained in $\pf[C]$; in particular, $\clf$ is uniquely determined by $f$.
Using this, we can show the following:

\begin{prop}
	The composition of two basic continuous maps is a basic continuous map.
\end{prop}

\begin{proof}
	Let $f\colon \mathcal{S}_1\to \mathcal{S}_2$ and $g\colon \mathcal{S}_2\to \mathcal{S}_3$ be basic continuous maps. Clearly, $gf$ satisfies (BC1).
To check (BC2), let
$D\in \mathcal{C}(\mathcal{S}_3)$. Since $D$ and $\clg [D]$ are both closed, we may use (BC3) on $g$ to conclude that
	$$\neg f^\cf_\leftarrow[\clg[D]]\sue f^{-1}[\neg \clg [D]]\sue f^{-1}[\pg[\neg D]].$$
	Then, by property (BC1), for any $C\in\mathcal C(\mathcal S_1)$,
	\begin{align*}
		g[f[C]]\sue D & \Rightarrow C\sue f^{-1}[\pg[D]]\\
		& \Rightarrow C\cap f^{-1}[\pg[\neg D]]\sue f^{-1}[\pg[D]]\cap f^{-1}[\pg[\neg D]]\\
		& \Rightarrow C\cap \neg f^{\cf}_\leftarrow[\clg[ D]]= \OS\\
		& \Rightarrow C\sue f^{\cf}_\leftarrow\clg[ D].
	\end{align*}
	On the other hand,
	$$C\sue f^{\cf}_\leftarrow[\clg[ D]]\Rightarrow C\sue f^{-1}[\clg[ D]]\Rightarrow C\sue f^{-1}[\pg[ D]]\Rightarrow f[g[C]]\sue D.$$
    This shows that $gf$ satisfies (BC2) with $(gf)^\cf_\leftarrow=f^\cf_\leftarrow\clg$.

    Properties (BC3) and (BC4) follow immediately from the last identity  $(gf)^\cf_\leftarrow=\clf\, g^\cf_\leftarrow$.
\end{proof}

It follows easily that basic zero-dimensional spaces, together with basic continuous maps, form a category. We denote it by \[\Bds.\]

\subsubsection*{\bf Example.}
The category $\loc$ of locales and localic maps is (non-fully) embeddable in $\Bds$.

\bigskip
In order to provide more examples, we need the following auxiliary results.

\begin{prop}\label{lema cl top}
	Let $\mathcal{S}=(X, \mathcal{C}(\mathcal{S}), \mathcal{B}d(\mathcal{S}))$ and $\mathcal{T}=(Y, \mathcal{C}(\mathcal{T}), \mathcal{B}d(\mathcal{T}))$. If $\bds$ and $\bdt$ are topological closure spaces, then a map  $f\colon X \to Y$ is a basic continuous map between $\mathcal{S}$ and $\mathcal{T}$ if and only if $f^{-1}[C]\in \mathcal{C}(\mathcal{S})$ for every $C\in \mathcal{C}(\mathcal{T})$.
\end{prop}

\begin{proof}
	Suppose that $(f, \clf)$ is a $\cf$-basic continuous pair and consider $C\in \mathcal{C}(\mathcal{T})$. By property (BC3),
	$$\pf[C]=X\smallsetminus\pf[Y\smallsetminus C]\sue X\smallsetminus (X\smallsetminus \clf[C])=\clf[C].$$
Since the converse inclusion holds always, we have $\pf[C]=\clf[C]\in \mathcal{C}(\mathcal{S})$.
	
	Conversely, it is straightforward to see that $f$ is basic continuous (just take $\clf[C]=\pf[C]$).
\end{proof}

The following proposition is obvious.

\begin{prop}\label{prop3.7}
	Let $X$ be a set, $\mathcal{F}\sue \mathcal{P}( X)$ with $\emptyset, X\in \mathcal{F}$, and $$\mathcal B=\{C\cup (X\smallsetminus D) \mid C, D \in \mathcal F \}.$$ Then
	$(X,\mathcal{F},\Omega \mho^\mathfrak{f}(\mathcal{B}))$
	is a basic zero-dimensional space.
\end{prop}

\begin{exams}\label{ex closure space}
	(1) Recall that a topological space $(X,\tau)$ is {\em zero-dimensional} if $\tau$ has a basis consisting of clopen sets \cite{Eng}.
		Each zero-dimensional topological space $(X,\tau)$ induces a canonical b0-space $$(X,\mathcal{C},\tau^{\,\mathrm c})$$ with $\mathcal C$ the set of all clopen subsets of $X$ and $\tau^{\,\mathrm c}$ the set of all closed subsets.  It follows from  Proposition \ref{lema cl top} that the category of zero-dimensional topological spaces is (non-fully) embeddable in $\Bds$.

\smallskip\noindent (2)
		Each closure space $(X, \mathcal{C})$ induces a b0-space $$(X,\mathcal{C},\Omega \mho^\mathfrak{f}(\mathcal{B}))$$ where $\mathcal B=\{C\cup (X\smallsetminus D) \mid C, D \in \mathcal C \}$ (by Proposition \ref{prop3.7}).
		It then follows from Proposition \ref{lema cl top} that the category of closure spaces is fully embeddable in $\Bds$.
		In particular, since $\TOP$ is equivalent to the category of topological closure spaces, it is then also fully embeddable in $\Bds$.

\smallskip\noindent (3)
Using similar arguments, one can show that
 the category of measurable spaces and measurable functions and
 the category of posets and right adjoint maps are both fully embeddable in $\Bds$.
\end{exams}

\begin{rema}
As mentioned previously, the category of b0-spaces and basic continuous maps was first proposed in \cite{EPP}. Our definition above is slightly different:
\begin{itemize}
\item
In the definition of a b0-space, we take $\mathcal{B}d$ as $\Omega\rotatebox[origin=c]{180}{$\Lambda$}^\mathfrak{f}_{\mathcal{B}d}(\mathcal{B})$ (instead of $\Omega(\mathcal{B})$), in order to guarantee that $\mathcal{B}d$ is a coframe.
\item
	In the definition of a basic continuous map $f\colon\mathcal{S\to \mathcal{T}}$, we do not demand it  to satisfy $\pf [C]\in \mathcal{C}(S)$ for every $C\in \mathcal{C}(\mathcal{T})$. Instead, we require the existence of a map $\clf\colon \mathcal{C}(\mathcal{T})\to\mathcal{C}(\mathcal{S})$ and condition (BC4).
    With these adjustments, we will get the crucial Lemma \ref{adj nos fechados} and the duality principle (see Note \ref{rema dual}).
    \end{itemize}
\end{rema}

\section{Basic images and preimages}

	As it should be clear from the examples above, the image of a  closed domain by a basic continuous map does not need to be  closed, not even a basic domain
    (consider e.g. $\mathcal{S}=\mathcal T=(\mathbb{R}, \{\emptyset, \mathbb R\}, \{\emptyset, \mathbb R\})$ and the map $f\colon \mathbb R \to \mathbb R$ defined by $f(x)=1$).
	Nevertheless, since $\mathcal{B}d$ is a closure system, Proposition \ref{sistemas e operadores de fecho} guarantees that it is associated with the closure operator
	\begin{equation*}
\langle Z \rangle =\tbigcap \{S\in \mathcal{B}d \mid Z\subseteq S\}\quad \mbox{ for all } Z \subseteq X.
	\end{equation*}
	Using this, we may associate to each basic continuous map $f\colon \mathcal{S}\to \mathcal{T}$ the map
	\[
			\rf= (M \mapsto  \langle f[M] \rangle)\colon \mathcal{B}d(\mathcal{S})  \to \mathcal{B}d(\mathcal{T})
\]
	which will be referred to as the {\em basic image of $f$}.

We will show now that we can extend the map $\clf$ defined in the closed domains to a map $\lf$ defined in arbitrary basic domains.
This map will be the b0-analogue of the localic preimage and will play a crucial role in the theory.

We will present the results in their most general form, but will prove them only for the case $\mathcal{B}d=\Omega (\mathcal{B})$, that is, where
 every basic domain $M\in \mathcal{B}d$ is of the form
$$M=\tbigcap_{i\in J}(C_i\vee \neg D_i)$$
for some $C_i, D_i \in \mathcal{C}$.
The proofs can be easily adapted to the general case $\mathcal{B}d=\Omega\rotatebox[origin=c]{180}{$\Lambda$}^\mathfrak{f}(\mathcal{B})$, they only require heavier
notation.

\begin{lemma}\label{adj nos fechados}
	Let $f\colon \mathcal{S}\to \mathcal{T}$ be a basic continuous map. Then, for every $M \in \mathcal{B}d(\mathcal{S})$ and every $C_i, \,D_i\in\mathcal{C}(\mathcal{T})$, $i=1,\ldots,n$,
	$$f[M]\sue \tbigvee_{i=1}^n( C_i\vee \neg D_i)\ \Leftrightarrow\ M \sue  \tbigvee_{i=1}^n( \clf[ C_i]\vee \neg \clf [D_i]).$$
\end{lemma}

\begin{proof}
	We will prove the result for $n=1$ only, but the arguments extend easily to any $n\in\mathbb N$.

	By (BC1) and (BC3) we have
	\begin{align*}
		f[M]\sue C\vee \neg D &\Rightarrow M\sue \pf[C\vee \neg D]\\
		& \Rightarrow M \cap \pf[\neg C \cap D] \sue \pf[C\vee \neg D]\cap \pf[\neg C \cap D]  \\
		& \Rightarrow M \cap \pf[\neg C] \cap \pf[D] \sue \OS \\
		& \Rightarrow M \cap (\neg \clf[ C] \cap \clf[D]) \sue \OS \\
		& \Rightarrow M \sue  \clf[ C] \vee \neg \clf[D].
	\end{align*}
	On the other hand, by property (BC4),
	$$M \sue  \clf[ C] \vee \neg \clf[D]\Rightarrow M \sue \pf [C\vee \neg D] \Rightarrow f[M]\sue C\vee \neg D. \qedhere $$
\end{proof}

\begin{prop}\label{prop adjunto de rf}
	For any basic continuous map $f\colon \mathcal{S} \to \mathcal{T}$, its basic image $\rf$ has a right adjoint.
\end{prop}

\begin{proof}
	By property (G3) of Galois adjunctions, it suffices to show that $\rf$ preserves arbitrary joins.
	
	Let $M_i\in \mathcal{B}d(\mathcal{S})$ ($i\in J$) and $C, D \in \mathcal{{C}(\mathcal{T})}$. By the previous lemma,
	\begin{align*}
		\rf[\tbigvee_{i\in J} M_i]\sue C\vee \neg D &\Leftrightarrow f[\tbigvee_{i\in J} M_i]\sue C \vee \neg D\\
		&\Leftrightarrow \tbigvee_{i\in J} M_i\sue \clf [C] \vee \neg \clf [D]\\
		&\Leftrightarrow  M_i\sue \clf [C] \vee \neg \clf [D]\ \text{ for all $i\in J$}\\
		&\Leftrightarrow  \rf [M_i]\sue  C \vee \neg D\ \text{ for all $i\in J$}\\
		&\Leftrightarrow  \tbigvee_{i\in J} \rf [M_i]\sue  C \vee \neg D.
	\end{align*}
	Hence $ \rf\left [\tbigvee_{i\in J} M_i\right]=\tbigvee_{i\in J}\rf [M_i]$ since $\mathcal{B}(\mathcal{T})$ is a meet-base of $\mathcal{B}d(\mathcal{T})$.
\end{proof}

Given a basic continuous map $f\colon \mathcal{S} \to \mathcal{T}$, we will denote the right adjoint of $\rf\colon \mathcal{B}d(\mathcal{S})  \to \mathcal{B}d(\mathcal{T})$ by $\lf$ and refer to it as the  {\em basic preimage of $f$}.

To obtain an explicit formula for $\lf[K]$ we first take $C_i, D_i\in \mathcal{C}(\mathcal{S})$ such that
$K=\tbigcap_{i\in J}(C_i\vee \neg D_i)$
and apply Lemma \ref{adj nos fechados} to conclude that, for every $M\in \mathcal{B}d(\mathcal{S})$,
\begin{align*}
	M\sue \lf[K]&\Leftrightarrow \rf [M]\sue K\\
	&\Leftrightarrow \rf [M]\sue C_i\vee \neg D_i\ \text{ for all $ i \in J$}\\
	&\Leftrightarrow M \sue \clf [C_i]\vee \neg \clf[D_i]\ \text{ for all $ i \in J$}\\
	&\Leftrightarrow M \sue \tbigcap_{i\in J} (\clf [C_i]\vee \neg \clf[D_i]).
\end{align*}

Then, immediately:

\begin{prop}\label{formula da preimagem}
	For each basic continuous map $f\colon \mathcal{S} \to \mathcal{T}$ and each basic domain \[K=\tbigcap_{i\in J}(C_i^1\vee \neg D_i^1\vee \dots \vee C_i^{n_i}\vee \neg D_i^{n_i}),\] with $C^1_i, D^1_i, \ldots, C_i^{n_i}, D_i^{n_i}\in \mathcal{C}(\mathcal{T})$ for every $i\in J$,
	$$\lf [K]=\tbigcap_{i\in J} \bigl(\clf[C_i^1]\vee \neg \clf[D_i^1]\vee \dots \vee \clf[C_i^{n_i}]\vee \neg \clf[D_i^{n_i}]\bigr).$$
\end{prop}

\section{A characterization of basic continuous maps}

Now that the basic preimage $\lf$ has been defined in all basic domains, we can use it to describe basic continuous maps.

\begin{theo}[Extension]\label{teo ext}
	 A map $f\colon |\mathcal{S}|\to |\mathcal{T}|$ between b0-spaces is basic continuous if and only if it satisfies the following properties:
	\begin{enumerate}[\em (BC1')]
		\item $\pf[\OS]=\OS$.
		\item There exists a map $f_\leftarrow\colon \mathcal{B}d(\mathcal{T})\to\mathcal{B}d(\mathcal{S})$ such that, for every $M \in \mathcal{B}d(\mathcal{S})$ and $ K\in \mathcal{B}d(\mathcal{T})$,  $$f[M]\subseteq K \mbox{ iff } M\subseteq \lf[K].$$
		\item For every $C\in\mathcal{C}(\mathcal{T})$,
		$\lf[C]\in \mathcal{C}(\mathcal{S})$.
		\item For every $C\in \mathcal{C}(\mathcal{T})$,  $\lf[\neg C]=\neg \lf[C]$.
	\end{enumerate}
	Moreover, the map $\lf$ in \textup(BC2'\textup)  is precisely the basic preimage of $f$.
\end{theo}

\begin{proof}
	Clearly, any basic continuous maps  $f\colon\mathcal{S}\to\mathcal{T}$ satisfies (BC1'). Moreover, by  propositions  \ref{prop adjunto de rf} and
 \ref{formula da preimagem}, its basic preimage $\lf$ also satisfies conditions (BC2')-(BC4').
	
	Conversely, consider the restriction to $\mathcal{C}(\mathcal{T})$ of the map $\lf$ given by (BC2'):
	\[
			\overline f_{\leftarrow}=(C\mapsto \lf[C]) \colon  \mathcal{C}(\mathcal{T})  \to \mathcal{C}(\mathcal{S}).
	\]
	By (BC3') it is well defined.
	Let us show that $(f, \overline f_{\leftarrow})$ is a $\cf$-basic continuous pair.
	
	Clearly $(f,\overline f_{\leftarrow} )$ satisfies conditions (BC1) and (BC2). Since (BC2') implies $\lf[K]\sue \pf[K]$, $(f,\overline f_{\leftarrow} )$ also satisfies  (BC3).
	It remains to check property (BC4). We will only do it for the case $n=1$, as the argument can be immediately extended to any $n\in\mathbb N$:
	
	By (BC2') and \eqref{propriedade do fecho}
	we have an adjunction $\rf \dashv \lf$.
    Since $\lf$ is a right adjoint, it must be order-preserving. Hence
	\begin{equation*}
		\overline f_{\leftarrow} [C]\vee \neg\overline f_{\leftarrow} [D]= \lf[C]\vee \lf[\neg D]\subseteq \lf[C\vee \neg D]\sue\pf[C\vee \neg D].
	\end{equation*}
	
	The fact that $\lf$ is indeed the basic preimage follows from the facts that $f$ is basic continuous and $\rf \dashv \lf$.
\end{proof}

We say that $(f, \lf)$ is a {\em basic continuous pair} to indicate that $f$ is basic continuous, with basic preimage $\lf$.
The basic preimage is the b0 analogue of the localic preimage. The preceding theorem states that a map admits a ``localic preimage" if and only if it is basic continuous.

\section{The basic preimage as a coframe homomorphism}

A fundamental property of localic topology is that in the category of locales the localic preimage $f_{-1}\colon \SL(M)\to\SL(L)$
of a localic map $f\colon L \to M$ is a coframe homomorphism (i.e., it preserves arbitrary meets and finite joins; in particular, it preserves complements).
In this section, we extend this result to b0-spaces by demonstrating that $\lf$ is a coframe homomorphism whenever $(f,\lf)$ is a basic continuous pair.

%The arguments used here are a slight adaptation of those presented in \cite{PP} for the case of locales and localic maps.

\begin{theo}\label{lf e coframe homo}
	Let $(f, \lf)\colon \mathcal S\to \mathcal T$ be a basic continuous pair between b0-spaces. Then $\lf$ is a coframe homomorphism.
\end{theo}

\begin{proof}
	Since $\lf$ is a right adjoint,  it preserves arbitrary meets.
	In order to show that it preserves binary joins, consider $K,\, L\in \mathcal{B}d(\mathcal{T})$ and $C_i, D_i, E_j, F_j \in \mathcal{C}(\mathcal{T})$, $i\in I, j\in J$, such that
	$$K=\tbigcap_{i\in I}(C_i\vee \neg D_i) \qtq{and} L=\tbigcap_{j\in J}(E_j \vee \neg F_j).$$
	Then use Proposition  \ref{formula da preimagem} and the fact that basic domains form a coframe to obtain
	\begin{align*}
		\lf[K\vee L]&= \lf [\tbigcap_{i, j} (C_i\vee \neg D_i \vee E_j \vee \neg F_j)]
		\\
		&= \tbigcap_{i, j} (\lf [C_i]\vee \neg \lf [ D_i] \vee \lf [E_j] \vee\neg \lf [F_j])\\
		&= \tbigcap_{i\in I} (\lf [C_i]\vee \neg \lf [ D_i]) \vee \tbigcap_{j\in J} (\lf [E_j] \vee\neg \lf [F_j])\\
		&=  \lf [\tbigcap_{i\in I}(C_i\vee \neg  D_i)] \vee  \lf [\tbigcap_{j\in J}(E_j \vee\neg F_j)]=\lf[K]\vee \lf[L].
	\end{align*}
Moreover,
	$\lf [\OS]\sue \pf [\OS]=\OS,$
	so $\lf$ also preserves the empty join.
\end{proof}

\begin{coro}
	Let $(f, \lf)$ be a basic continuous pair. Then $\lf$ preserves complements.
\end{coro}

\section{An application to Heyting semilattices}

We will now apply the theory of b0-spaces to address the issues encountered in \cite{EPP} when attempting to generalise certain results from locales to Heyting semilattices.

\bigskip
	Let $H$ be a Heyting semilattice and set $$\mathcal{B}=\{\cf a \vee\of b \mid a, b \in H\}$$ (where the joins are taken in $\mathcal{I}dl(H)$). The {\em basic domains}  of $H$ are the elements of $\mathcal{B}d(H)=\Omega \rotatebox[origin=c]{180}{$\Lambda$}_{\mathcal{I}dl(H)}^\mathfrak{f}(\mathcal{B}).$

By definition, basic domains are the subsets of $H$ of the form
$$\tbigcap_{i\in J}(\cf a_i^1\vee \of b_i^1\vee \dots \vee \cf a_i^{n_i}\vee \of b_i^{n_i}).$$
Obviously, $\of a\in \mathcal{B}d(H)$ for every $a\in H$. This is also true for closed sets:

\begin{prop}
For every $a\in H$, $\cf a \in \mathcal{B}d(H)$.
\end{prop}

\begin{proof}
It is clear that $\cf a \sue \tbigcap \{\cf a \vee \of b \mid b\in H\}. $ Conversely, let $x\in \tbigcap \{\cf a \vee \of b \mid b\in H\}$. Then there exist $a'\geq a$ and $c$ in $H$ such that $x=a'\wedge (x\wedge a)\to c$. By properties (H1), (H2) and (H4) in \ref{Heytsemil} we have
$$1=(x\wedge a)\to x= ((x\wedge a)\to a') \wedge ((x\wedge a)\to((x\wedge a)\to c))= (x\wedge a)\to c.$$ Therefore  $x=a'\geq a$ and
$\cf a=\tbigcap \{\cf a \vee \of b \mid b\in H\}$.
\end{proof}

Hence $\of a, \cf a \in \mathcal{B}d(H)$
and they are complemented to each other.

\bigskip
Moreover, it follows from Proposition \ref{prop tecnica} that the system $\mathcal{B}d(H)$ is a coframe.
Hence, the triple $$(H, \cf H, \mathcal{B}d(H))$$ is a basic zero-dimensional space.

In Heyting semilattices, basic domains generalize
nuclear ranges. To prove this, we need the following result from \cite[Prop. 2.15]{PPT}.

\begin{lemma}
	Let $\nu \colon H\to H$ be a nucleus on a Heyting semilattice $H$. For every $a, b \in H$,
	$$a\to \nu (b)=\nu(a)\to \nu (b)$$
\end{lemma}

The next result is an adaptation of the corresponding result in \cite{PP} for locales and localic maps.
The proof in \cite{PP} extends easily to Heyting semilattices by replacing $x\vee a$ by $(a\to x)\to x$ and using properties (H5) and (H6).

\begin{prop}\label{prop nuc em bd}
	Let $\nu \colon H \to H$ be a nucleus on a Heyting semilattice, and let $N_\nu\subseteq H$ be its range. Then $N_\nu$ is a basic domain of the  form
	$$
	N_\nu=\tbigcap\{\cf a \vee \of b \mid \nu(a)=\nu (b)\}.
	$$
\end{prop}

\begin{proof}
Consider $x\in N_\nu$ and $a,b \in H$ such that $\nu(a)=\nu (b)$.
Let $B=\tbigcap\{\cf a \vee \of b \mid \nu(a)=\nu (b)\}$. By the Lemma, we have
	$$a\to x=\nu(a)\to \nu (x)=\nu(b)\to \nu (x)=b\to x.$$
	Then, by properties (H5) and (H6),
	$$x=((a\to x)\to x) \wedge (a\to x)= ((a\to x)\to x) \wedge (b\to x)\in \cf a\vee \of b.$$
	Conversely, let $x\in B$. Since $ \nu (\nu (x))=\nu(x)$, we have $x\in \cf (\nu(x))\vee \of x$. Hence there are $a\geq \nu (x)$ and $c\in H$  such that $x=a\wedge (x\to c) $. Then, applying (H2), (H1) and (H4), we get
	$$1=x\to x=(x\to a)\wedge x\to (x\to c)=x\to c.$$
	Hence $x=a\geq \nu (x)\geq x$ and  $x=\nu (x)\in N_\nu$.
\end{proof}

This proposition makes it clear that basic domains of $H$ are precisely the sublocales of $H$ when $H$ is a complete Heyting lattice.

Regarding morphisms we have from \cite[Thm. 4.6]{EPP}:

\begin{theo}\label{pre image de l-morfismos}
	A map between Heyting semilattices $f\colon H_1\to H_2$ is continuous if and only if the following conditions hold:
	\begin{enumerate}[\em (1)]
		\item $f^{-1}[\OS]=\OS$.
		\item For every $b\in H_2$, $f^{-1}[\cf b]$ is closed, with  $\pf [\cf b]=\cf (f^\ast (b))$.
		\item For every $b\in H_2$, $\neg \pf [\cf b]\sue f^{-1}[\of b]$, where the complement is taken in $\mathcal{I}dl(H_1)$.
	\end{enumerate}
\end{theo}

\begin{prop}\label{continos sao basic}
	Let $f\colon H_1\to H_2$ be a continuous map between Heyting semilattices. Then $f$ is a basic continuous map $$(H_1, \cf H_1, \mathcal{B}d(H_1))\to (H_2, \cf H_2, \mathcal{B}d(H_2)).$$

\end{prop}

\begin{proof}
	It is clear from the previous theorem that $f$ satisfies properties (BC1) to (BC3).
	Let us show that it also satisfies (BC4).  We will do this for $n=1$, but the argument extends easily to every natural number.
	
	By Theorem  \ref{pre image de l-morfismos}, it suffices to prove that $f[\cf (f^*(a))\vee \of (f^*(b))]\subseteq \cf a\vee \of b$.
	So let $x\in \cf (f^*(a))\vee \of (f^*(b))$. Then there are $z\geq f^*(a)$ and $c\in H_1$ such that $x=z\wedge (f^*(b)\to c)$. Applying the Frobenius identity from   Proposition \ref{prop carac cont} and the fact that $f(z)\geq f(f^*(a))\geq a$, we obtain
	\begin{align*}
		f(x)&=f(z\wedge (f^*(b)\to c))= f(z) \wedge f(f^*(b)\to c)\\&=f(z) \wedge( b\to f(c))\in \cf a\vee\of b. \qedhere
	\end{align*}
\end{proof}

\begin{coro}\label{coro pre imagem par hey}
	Let $f\colon H_1\to H_2$ be a continuous map between Heyting semilattices. For each $K\in \mathcal{B}d(H_2)$, there exists the largest basic domain contained in $\pf[K]$, denoted $\lf[K]$. Moreover, the map
	\[
			\lf=(K\mapsto \lf[K]) \colon \mathcal{B}d(H_2)  \to \mathcal{B}d(H_1)
	\]
	is a coframe homomorphism satisfying $\lf[\of b]=\of (f^\ast(b))=\neg \pf [\cf b]$.
\end{coro}

\begin{prop}
	A map between Heyting semilattices is continuous if and only if it is basic continuous and monotone.
\end{prop}

\begin{proof}
	Let $f\colon H_1 \to H_2$ be a monotone basic continuous map between Heyting semilattices. Then, for every $a\in H_1$ and $b\in H_2$,
	$$f[\cf a]\sue \cf b  \Leftrightarrow f(a)\leq b \Leftrightarrow f(a)\in \cf b.$$	
	We can now use condition (BC2) to get
	$$a \in \clf[\cf b] \Leftrightarrow \cf (a)\sue \clf[\cf b]\Leftrightarrow f[\cf a]\sue \cf b \Leftrightarrow  f(a) \in \cf b,$$
	and thus $\clf[\cf b]=\pf[\cf b]$. Hence $\pf[\cf b]$ is always closed and by (BC3) it satisfies $\neg \pf [\cf b]\sue \pf[\of b]$ and, by Theorem \ref{pre image de l-morfismos}, $f$ is basic continuous.
	
	The converse follows immediately from Proposition \ref{continos sao basic} and the fact that every continuous map is monotone (since it is a right adjoint).
\end{proof}

\begin{rema}
A basic continuous map between Heyting semilattices is not necessarily monotone (and hence continuous), as some easy examples show.
\end{rema}

\begin{note}
    The effectiveness of the theory of b0-spaces can be observed by comparing the results above about Heyting semilattices with those obtained when one takes the standard approach to generalized sublocales as nuclear ranges. Indeed, it was shown in \cite{EPP} that the system of nuclear ranges may not even be a lattice, while the system of basic domains is always a coframe; moreover, there is no non-complete analogue for the localic preimage when one uses nuclear ranges. Here, it not only exists, but it even provides a coframe homomorphism.
\end{note}

\section{Basic open maps}\label{sec open}

	We say that a basic continuous map $f\colon \mathcal{S}\to \mathcal{T}$ is  a \emph{basic open map} if, for every $U\in \mathcal{O}(\mathcal{S})$, $\rf[U]$ is still an open domain.

\begin{note}\label{rema basic open =open} One might ask why it is the set $\rf[U]$ rather than $f[U]$ that is required to be open.
	The reason for this is the better properties of $\rf$, namely the adjunction $\rf \dashv\lf$.  But this has a consequence: our general notion does not need to coincide with the standard definition of open maps in certain subcategories of $\Bds$. Nevertheless, in some specific settings they do coincide:
	
	\begin{itemize}
		\item $T_D$-closure spaces: As we have seen in Example 2 in  \ref{ex closure space}, the category of closure spaces can be fully embedded in $\Bds$ by sending each closure space $(X, \mc)$ to $(X, \mc, \Omega\mho^\mathfrak{f}(\mathcal{B}))$, with $\mathcal{B}=\{C\cup X\setminus D\mid C, D\in \mc \}$. In case $(X, \mathcal{C})$ is a {\em $T_D$-closure space} (that is, $\overline{\{x\}}\smallsetminus \{x\}$ is closed for every $x\in X$) then, for every $x\in X$
	\begin{equation*}
		X\smallsetminus \{x\}=\overline{\{x\}}\smallsetminus \{x\}\cup X\smallsetminus \overline{\{x\}}\in \mathcal{B}.
	\end{equation*}
        Therefore $\Omega\mho^\mathfrak{f}(\mathcal{B})=\mathcal{P}(X)$.
       Hence $\rf$ is the set-theoretical image and the definitions of basic open map and open map do coincide.
		
In particular, this holds for $T_D$-topological spaces.

		\item Heyting semilattices and {\bf continuous maps}: Let $f\colon H_1\to H_2$ be a continuous map between Heyting semilattices. By  propositions \ref{abertos sao nuc} and  \ref{cont preservam nuc}, we know that $f[\of a]\in \mathcal{N}uc(H_2) \sue \mathcal{B}d(H_2)$. Therefore
		$f[\of a]=\rf[\of a]$ and basic open maps are precisely the open maps. However, this need not be the case when $f$ is a general basic continuous map.
	\end{itemize}
\end{note}

We will now present a characterization of basic open maps between basic zero-dimensional spaces.
First, we need to introduce some notation and terminology.

Given a b0-space $\mathcal{S}=(X, \mathcal{C}, \mathcal{B}d)$, set \[\mathcal{O}^\mathfrak{f}(\mathcal{S})=\rotatebox[origin=c]{180}{$\Lambda$}^\mathfrak{f}\Lambda^\mathfrak{f}\mathcal{O}(\mathcal{S}).\] We will refer to the elements of $\mathcal{O}^\mathfrak{f}(\mathcal{S})$ as {\em finitely generated open domains}.
Further, given a basic continuous pair $(f,\lf)$ between b0-spaces $\mathcal{S}$ and $\mathcal{T}$, we will refer to the map
\[
		\lf^{\,\mathfrak{f} \of} =(V\mapsto \lf[V]) \colon \mathcal{O}^\mathfrak{f}(\mathcal{T})  \to \mathcal{O}^\mathfrak{f}(\mathcal{S})
\]
as the {\em finite open restriction of $\lf$}. (Note that this map is well defined since the basic preimage of an open is still an open and $\lf$ preserves finite joins and finite meets.)

\begin{rema}\label{lema dos complementos} In any distributive lattice
	$L$, for each $a, b, c\in L$ with $c$ complemented,
	$$a\leq  b\vee \neg c\ \Leftrightarrow\ a\wedge c \leq b.$$
\end{rema}

Next lemma provides a useful Frobenius-type condition.

\begin{lemma}\label{lemma frobenius 1}
    Let $f\colon \ms \to \mt $ be a basic continuous map. For any $U\in \mos$ and $V_1, \ldots ,V_n \in \mot$,
    $$\rf[U\cap\lf[V_1 \cap \cdots\cap V_n]]=\rf[U]\cap V_1 \cap \cdots\cap V_n .$$
\end{lemma}
\begin{proof}
Let
$V=V_1 \cap \cdots\cap V_n.$
    From the adjunction $\rf\dashv \lf$, we can use the monotonicity of $f$ and condition (G1) to get
    \begin{align*}
        \rf[U\cap\lf[V]]&\sue\rf[U]\cap \rf[\lf[V]]\sue \rf[U]\cap V.
    \end{align*}
    On the other hand, since $\rf$ is a left adjoint and $\lf$ preserves complements, we have
\begin{align*}
    \rf[U]\sue \rf[U \vee \neg\lf[V] ]
    &= \rf[(U \cap \lf[V] ) \vee \neg\lf[V] ]\\
    &= \rf[U \cap \lf[V] ] \vee \rf[\lf[\neg V] ]\\
    &\sue \rf[U \cap \lf[V] ] \vee \neg V .
\end{align*}
Hence, by the basic property in the note above, $\rf[U]\cap V \sue \rf[U \cap \lf[V]]$.
\end{proof}

We are now able to prove our main result.

\begin{theo}\label{Joyal-Tierney basic}
Let $\ms$ be a b0-space such that $\Lambda^\mathfrak{f}\mos=\mos$ and let $\mt$ be an arbitrary b0-space.
	Let also $(f,\lf)$ be a basic continuous pair between $\mathcal{S}$ and $\mathcal{T}$. Then $f$ is a basic open map if and only if $\lf^{\,\mathfrak{f} \of}$ has a left adjoint $$h\colon \mathcal{O}^\mathfrak{f}(\mathcal{S})\to \mathcal{O}^\mathfrak{f}(\mathcal{T})$$ such that the following conditions hold:
	\begin{enumerate}[\em (1)]
		\item For every $U\in \mathcal{O}(\mathcal{S})$, $h[U]\in \mathcal{O}(\mathcal{T})$.
		\item For every $U\in \mathcal{O}(\mathcal{S})$ and $V_1, \ldots, V_n\in \mathcal{O}(\mathcal{T})$,
		$$h[U\cap \lf[V_1\cap \dots \cap V_n]]=h[U]\cap V_1\cap \dots \cap V_n.$$
	\end{enumerate}
\end{theo}
\begin{proof}
Suppose $f$ is a basic open map.
Let $U\in \mo^\mathfrak{f}(\ms)$. Since $\mo^\mathfrak{f}(\ms)=\rotatebox[origin=c]{180}{$\Lambda$}^\mathfrak{f}
\Lambda^\mathfrak{f}\mathcal{O}(\mathcal{S})=\rotatebox[origin=c]{180}{$\Lambda$}^\mathfrak{f}\mos$, either $U=\OS$, or there exist $U_1, \dots, U_n\in \mos$ such that $U=\tbigvee_{i=1}^n U_i$. In the first case $\rf[\OS]=\OS\in \mo^\mathfrak{f}(\mt)$. In the second case, since $\rf$ is a left adjoint, we have
$\rf[U]=\tbigvee_{i=1}^n \rf[U_i]\in \mo^\mathfrak{f}(\mt).$
    Hence  the mapping
    	\[
			 h=(U \mapsto \rf [U]) \colon \mo^\mathfrak{f} (\ms)  \to \mo^\mathfrak{f} (\mt)
	\]
    is well defined.
    Because $h$ is a restriction of $\rf$, it immediately follows that $h$ is the left adjoint of $\lf^{\mathfrak{f}\of}$ and that it satisfies condition (1). The fact that $h$ satisfies condition (2) follows immediately from Lemma \ref{lemma frobenius 1}.

    Conversely, suppose that $\lf^{\,\mathfrak{f} \of}$ has a left adjoint $h\colon \mathcal{O}^\mathfrak{f}(\mathcal{S})\to \mathcal{O}^\mathfrak{f}(\mathcal{T})$ satisfying conditions (1) and (2) and pick $U\in \mos$ and $W_1, \ldots, W_n, V_1, \ldots V_n\in \mot$. Let
    $W=  W_1\vee \cdots \vee W_n, \text{ and } V= V_1\cap \cdots\cap V_n. $
    Applying  \ref{lema dos complementos} and \ref{lemma frobenius 1} we get
    \begin{align*}
        \rf[U] \sue W\vee \neg V &\Leftrightarrow \rf[U] \cap V\sue W\\
        &\Leftrightarrow \rf[U \cap \lf[V]] \sue W
        \Leftrightarrow U \cap \lf[V] \sue \lf[ W].
    \end{align*}
    But $h\colon \mathcal{O}^\mathfrak{f}(\mathcal{S})\to \mathcal{O}^\mathfrak{f}(\mathcal{T})$ is the left adjoint of $\lf^{\,\mathfrak{f} \of}$,
    $U \cap \lf[V]\in \mo^\mathfrak{f}(\ms)  \mbox{ and } \ W\in \mo^\mathfrak{f}(T)$. Hence
        \begin{align*}
        \rf[U] \sue W\vee \neg V  &\Leftrightarrow U \cap \lf[V] \sue \lf[ W]\\&
        \Leftrightarrow U \cap \lf[V] \sue \lf^{\mathfrak{f}\of}[ W] \Leftrightarrow h[U \cap \lf[V]] \sue  W\\
        &\Leftrightarrow h[U] \cap V \sue  W \Leftrightarrow h[U] \sue  W\vee \neg V.
    \end{align*}
    Finally, notice that every element in $\mathcal{B}$ is of the form $$ W\vee \neg V=(W_1\vee \cdots \vee W_n) \vee \neg (V_1\cap \cdots\cap V_n),$$ for some $W_1, \ldots, W_n, V_1, \ldots V_n\in \mot$. Since $\mathcal{B}$ is a meet-base of $\bdt$, we conclude that
    $\rf[U]=h[U ]\in \mot.$
\end{proof}

\begin{rema}
The assumption $\Lambda^\mathfrak{f}\mos=\mos$ in this result is essential in the proof of `$\Rightarrow$'.
For example, the following Boolean algebras

\medskip
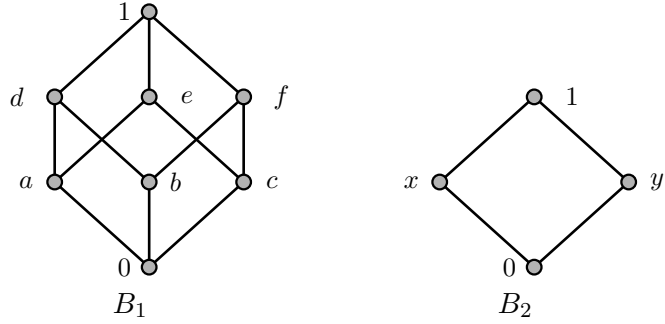
\begin{center}
\begin{tikzpicture}[scale=.5]
%%% edges
\draw[line width=0.35mm] (2.5,0) -- (2.5,2.25);
\draw[line width=0.35mm] (2.5,0) -- (0,2.25);
\draw[line width=0.35mm] (2.5,0) -- (5,2.25);
\draw[line width=0.35mm] (0,4.5) -- (0,2.25);
\draw[line width=0.35mm] (2.5,4.5) -- (0,2.25);
\draw[line width=0.35mm] (0,4.5) -- (2.5,2.25);
\draw[line width=0.35mm] (5,4.5) -- (2.5,2.25);
\draw[line width=0.35mm] (2.5,4.5) -- (5,2.25);
\draw[line width=0.35mm] (5,4.5) -- (5,2.25);
\draw[line width=0.35mm] (0,4.5) -- (2.5,6.75);
\draw[line width=0.35mm] (2.5,4.5) -- (2.5,6.75);
\draw[line width=0.35mm] (5,4.5) -- (2.5,6.7);
%% vertex labels
\node at (1.85,0) {\small 0};
\node at (1.85,6.75) {\small 1};
\node at (-.75,2.25) {\small$a$};
\node at (3.2,2.25) {\small$b$};
\node at (5.75,2.25) {\small$c$};
\node at (-1,4.5) {\small$d$};
\node at (3.5,4.5) {\small$e$};
\node at (6,4.5) {\small$f$};
\node at (2,-1) { $B_1$};
%% vertices
\draw[fill=black!30!white,line width=0.30mm] (2.5,0) circle (5.5pt);
\draw[fill=black!30!white,line width=0.30mm] (0,2.25) circle (5.5pt);
\draw[fill=black!30!white,line width=0.30mm] (2.5,2.25) circle (5.5pt);
\draw[fill=black!30!white,line width=0.30mm] (5,2.25) circle (5.5pt);
\draw[fill=black!30!white,line width=0.30mm] (0,4.5) circle (5.5pt);
\draw[fill=black!30!white,line width=0.30mm] (2.5,4.5) circle (5.5pt);
\draw[fill=black!30!white,line width=0.30mm] (5,4.5) circle (5.5pt);
\draw[fill=black!30!white,line width=0.30mm] (2.5,6.75) circle (5.5pt);
\end{tikzpicture}
\hspace{10mm}
	\begin{tikzpicture}[scale=.5]
		%%% edges
		\draw[line width=0.35mm] (2.5,1) -- (0,3.25);
		\draw[line width=0.35mm] (2.5,1) -- (5,3.25);
		\draw[line width=0.35mm] (2.5,5.5) -- (0,3.25);
		\draw[line width=0.35mm] (2.5,5.5) -- (5,3.25);
		%% vertex labels
		\node at (1.85,1) {\small$0$};
		\node at (-.75,3.25) {\small$x$};
		\node at (5.75,3.25) {\small$y$};
		\node at (3.5,5.5) {\small$1$};
        \node at (2,0) {$B_2$};
		%% vertices
		\draw[fill=black!30!white,line width=0.30mm] (2.5,1) circle (5.5pt);
		\draw[fill=black!30!white,line width=0.30mm] (0,3.25) circle (5.5pt);
		\draw[fill=black!30!white,line width=0.30mm] (5,3.25) circle (5.5pt);
		\draw[fill=black!30!white,line width=0.30mm] (2.5,5.5) circle (5.5pt);
	\end{tikzpicture}
    \end{center}
induce b0-spaces $$(B_1, \{\downarrow\! 0, \downarrow\! a, \downarrow\! b, \downarrow\! c, \downarrow\! 1\}, \mathfrak I_p(B_1))\ \mbox{ and }\ (B_2, \{\downarrow\! 0, \downarrow\! y , \downarrow\!1 \}, \mathfrak I_p(B_2))$$ (where $\mathfrak I_p(B_i)$ denotes the lattice of all principal ideals of $B_i$).
The map $g\colon B_1\to B_2$ given by
\begin{align*}
    &g(0)=0,\\
    &g(b)=y,\\
    &g(a)=g(c)=g(e)= x,\\
    &g(d)=g(f)=g(1)=1,
\end{align*}
is a basic open map but $g_\leftarrow^{\,\mathfrak{f} \of}$ does not have a left adjoint $h$ satisfying conditions (1) and (2). Indeed, we know from the proof of the theorem that if such an adjoint would exist, then $$g_\rightarrow[U]=h[U]\in \mathcal{O}^\mathfrak{f}(B_2), \mbox{ for all }U\in \mathcal{O}^\mathfrak{f}(B_1).$$ But $g_\rightarrow[\downarrow\! b] = \downarrow\! y \not \in \mathcal{O}^\mathfrak{f}(B_2)$ while $\downarrow\! b\in \mathcal{O}^\mathfrak{f}(B_1)$ (since $\downarrow\! b=\downarrow\! d\,\cap\downarrow\! f$).
\end{rema}

The following is an important case of the previous theorem that will be relevant when studying Heyting semilattices.

\begin{prop}\label{Joyal-Tierney basic simples}
Let $\ms$ and $\mt$ be b0-spaces such that $\Lambda^\mathfrak{f}\mos=\mos$ and $\Lambda^\mathfrak{f}\mot=\mot$.
	Let also $(f,\lf)$ be a basic continuous pair between  $\mathcal{S}$ and $\mathcal{T}$ such that, for every $U\in \mathcal{O}(\mathcal{S})$,
	$$\rf[U]=\tbigcap \{ \neg V\vee W \mid \rf[U]\sue  \neg V\vee W \text{ and } V, W \in \mathcal{O}(\mathcal{T})\}.$$
	Then $f$ is a basic open map if and only if the map
	\[
			\lf^\of=(V\mapsto \lf[V]) \colon \mathcal{O}(\mathcal{T})  \to  \mathcal{O}(\mathcal{S}) 		
	\]
	has a left adjoint $h\colon \mathcal{O}(\mathcal{S})\to\mathcal{O}(\mathcal{T})$ such that
		$$h[U\cap \lf[V]]=h[U]\cap V \quad \mbox{for every } U\in \mathcal{O}(\mathcal{S}) \mbox{ and } V\in \mot.$$
\end{prop}

\begin{proof}
    $\Rightarrow$: Take $h[U]=\rf[U]$ for all $U\in \mos$.

\noindent
    $\Leftarrow$: Repeating the arguments from the previous theorem, we can conclude that
    $$\rf[U]\sue \neg V\vee W \Leftrightarrow h[U]\sue \neg V\vee W \quad\mbox{for all } V, W\in \mot. $$
    Hence $h[U]\sue \tbigcap \{ \neg V\vee W \mid \rf[U]\sue  \neg V\vee W \}=\rf[U]$.

    For the converse inclusion, observe that from $h\dashv \lf^\of$, $\rf\dashv\lf$ and $h[U]\in \mos$, it follows that
    $$h[U]\sue h[U]\Rightarrow U\sue \lf^\of h[U]\Rightarrow U\sue \lf h[U]\Rightarrow \rf[U]\sue  h[U].\qedhere $$
\end{proof}

\begin{rema}
Theorem \ref{Joyal-Tierney basic} and Proposition \ref{Joyal-Tierney basic simples} are special cases of more general results for distributive lattices proved in \cite{JA}.
\end{rema}

\section{An application: open maps between Heyting semilattices}

Open maps in the category of locales are characterized by the well-known \emph{Joyal-Tierney Theorem} \cite{JT84}:

\begin{theo}
		A localic map $f \colon L \to M$ is open if and only if its adjoint frame homomorphism
		$f^*\colon M\to L$ is a complete Heyting homomorphism \textup(i.e., it also
		preserves arbitrary meets as well as the Heyting operation\textup).
\end{theo}

This result was extended to the non-complete case of Heyting semilattices in \cite{EPP}.
The main goal of this section is to demonstrate how Proposition \ref{Joyal-Tierney basic simples} can be used to obtain a new proof of the Joyal-Tierney Theorem for Heyting semilattices.

We start by recalling the following lemma from \cite{EPP}.

\begin{lemma}\label{lema adj jt}
	Let $f\colon H_1\to H_2$, $f^*\colon H_2\to H_1$ and $f_!\colon H_1\to H_2$  be maps between Heyting semilattices such that $f_!\dashv f^*\dashv f$. The following are equivalent:
	\begin{enumerate}[{\em (i)}]
		\item $f_!(a \wedge  f^\ast(b))=f_!(a)\wedge b$ for every $a\in H_1$ and every $b\in H_2$.
		\item $f^*(b\to c)=f^*(b)\to f^*(c)$ for every $b, c\in H_2$.
	\end{enumerate}
\end{lemma}

\begin{theo}\label{JTtheo}
	The following are equivalent
	for a continuous map $f\colon H_1 \to H_2$ between Heyting semilattices:
	\begin{enumerate}[\em(i)]
		\item $f$ is open.
		\item $f$ is basic open.
		\item The map $ \lf^\of =(\of b\mapsto \lf[\of b] ) \colon \of H_2 \to  \of H_1 $
		has a left adjoint $$h\colon\of H_1\to \of H_2$$ such that
		$h[\of a\cap \lf[\of b]]=h[\of a]\cap \of b$ for every $a\in H_1$ and $b\in H_2$.
		\item $f^*$ has a left adjoint
		$f_{!}\colon H_1\to H_2$ such that
		$
		f_!(a\wedge f^*(b))=f_!(a)\wedge b$ for every $a\in H_1$ and $b\in H_2$.
		\item $f^*$ is a right adjoint such that
		$f^\ast (b\to c)=f^\ast (b)\to f^\ast (c)$
		 for every $b,c\in H_2$.
	\end{enumerate}
\end{theo}

\begin{proof}
	(i)$\Leftrightarrow$(ii): Follows immediately from Remark \ref{rema basic open =open}.
	
	\smallskip
	\noindent
	(ii)$\Leftrightarrow$(iii): Let $a\in H_1$. We know that $f[\of a]\in \mathcal{N}uc(H_2)$. Hence, by Proposition \ref{prop nuc em bd},
	$$\rf [\of a]=\tbigcap\{\rf [\of a]\sue \cf x \vee \of y \mid x, y \in H_2\}.$$
    Further, in any Heyting semilattice $H$, $\Lambda^\mathfrak{f}\of (H)=\of H$.
	Therefore, we can apply Proposition \ref{Joyal-Tierney basic simples} and the result follows.

	\noindent
	(iii)$\Leftrightarrow$(iv):  Since $f$ is  continuous,  $\lf^\of(\of b)=\of (f^\ast (b))$. Moreover, since
\[
\of = (a \mapsto \of a)\colon H  \to \of H
\]
	is an isomorphism  for every Heyting semilattice $H$,
	there exists a one-to-one correspondence between maps $h\colon\of H_1\to \of  H_2$ and maps $f_!\colon H_1\to  H_2$ such that for each $h$ the associated  $f_!$ makes the following diagram to commute:
\[
\xymatrix@R=30pt@C=40pt{\of H_1  \ar[r]^{h} & \of  H_2 \ar[r]^(.55){\lf^\of} & \of H_1   \\
H_1 \ar[u]^{\of} \ar[r]^{f_!} & H_2 \ar[u]^{\of} \ar[r]^(.5){f^*} &  H_1. \ar[u]_{\of}}
\]
	Now, we use the fact that $\of\colon H\to \of H$ is an isomorphism and the commutativity of the previous diagram to deduce that $\lf^\of$ has a left adjoint $h\colon\of H_1\to \of  H_2$  if and only if $f^\ast\colon H_1\to  H_2$  has a left adjoint $f_!\colon H_1\to  H_2$:
\[
\xymatrix@R=30pt@C=30pt{h    & \dashv & \lf^\of    \\
f_! \ar[u]^{\of}   & \dashv     &  f^*. \ar[u]_{\of}}
\]
	Moreover, by the commutativity of the diagrams,
	$$\of  (f_!(a\wedge f^\ast(b))=h[\of a \cap \lf^\of[\of b]] \qtq{and} \of (f_!(a)\wedge b)=h[\of a]\cap \of b.$$
	Hence
	$$f_!(a\wedge f^\ast(b))=f_!(a)\wedge b \Leftrightarrow h[\of a \cap \lf^\of[\of b]] =h[\of a]\cap \of b.$$

	Finally, the equivalence
	(iv)$\Leftrightarrow$(v) follows from Lemma \ref{lema adj jt}.
    \end{proof}

\section{Basic closed maps and a duality principle for b0-spaces}\label{dual}

	We say that a basic continuous map $f\colon \mathcal{S}\to \mathcal{T}$ between b0-spaces is \emph{basic closed} if $\rf[B]\in \mathcal{C}(\mathcal{T})\mbox{ for every }B\in \mathcal{C}(\mathcal{S}).$

\begin{rema}\label{rema basic closed=closed}
	As with basic open maps, it is $\rf[C]$ that we require to be closed, not $f[C]$. This implies that the definition of a basic closed map may not coincide with that of a closed map in some subcategories of $\Bds$. Again, they do coincide in
$T_D$-closure spaces, as well as in
		  Heyting {\bf lattices} and continuous maps.
Indeed, let
 $f\colon H_1 \to H_2$ be a continuous map between Heyting lattices. By propositions  \ref{abertos sao nuc} and \ref{cont preservam nuc}, $f[\cf a] \in \mathcal{N}uc(H_2)\sue \mathcal{B}d(H_2)$. Hence
		$\rf[\cf a]=f[\cf a]$ and basic closed maps are precisely the closed maps.

This is not necessarily true when $H$ is not a Heyting \emph{lattice} because in Heyting semilattices  $\cf a$ is not necessarily a nuclear range.
\end{rema}

Our aim is now to demonstrate how Theorem \ref{Joyal-Tierney basic} can be used to derive a characterization of closed maps. To achieve this, we need to introduce the notion of a dual space.

Giving a b0-space $\mathcal{S}=(X,\mathcal{C}( \mathcal{S}),  \mathcal{B}d (\mathcal{S}))$, let  $$\widetilde{\mathcal{S}}=(X, \mathcal{O}( \mathcal{S}),  \mathcal{B}d (\mathcal{S})),$$ which we call the {\em dual of $\mathcal{S}$}, be the b0-space whose closed domains  are precisely the open domains of $\mathcal{S}$.

Given a morphism $f\colon\mathcal{S}\to \mathcal{T}$ between b0-spaces, we denote by $|f|$ the map between their underlying sets. Clearly, $f\colon\mathcal{S}\to \mathcal{T}$ is totally determined by $|f|$. So we can define $\widetilde f\colon\widetilde{\mathcal{S}}\to \widetilde{\mathcal{T}}$ such that $|\widetilde f|=|f|$.

%It is an easy exercise to check that:
\begin{lemma}
	Let $(f, \lf)$ be a basic continuous pair between b0-spaces $\mathcal{S}$ and $\mathcal{T}$, and let
	\[
			\widetilde f_\leftarrow=(K \mapsto \lf [K]) \colon   \mathcal{B}d(\widetilde{\mathcal{T}})  \to  \mathcal{B}d(\widetilde{\mathcal{S}}).
	\]
	Then $(\widetilde f, \widetilde f_\leftarrow)$ is a basic continuous pair between $\widetilde{\mathcal{S}}$ and $\widetilde{\mathcal{T}}$.
\end{lemma}

\begin{proof}
	It is clear that $(\widetilde f,\widetilde f_\leftarrow)$ satisfies (BC1') and (BC2').
	
	Let $C\in \mathcal{C}(\widetilde{\mathcal{S}})$. Then $C$ is an open of $\mathcal S$ and as so there exists a $D\in \mathcal C (\mathcal{S})$ such that $C=\neg D$. Now observe that $\widetilde f_\leftarrow[C]=\lf[\neg D]=\neg \lf [D ]\in \mathcal{C}(\widetilde{\mathcal{S}})$, hence $(\widetilde f, \widetilde f_\leftarrow)$ enjoys property (BC3'). Moreover $$\neg \widetilde f_\leftarrow[C]=\neg \lf[\neg D]=\lf [D ]=\widetilde f_\leftarrow[\neg C],$$ so  $(\widetilde f,\widetilde f_\leftarrow)$ also satisfies (BC4').
\end{proof}

It is straightforward to check that the assignments $S\mapsto \widetilde S$ and $f\mapsto \widetilde f$ establish a (self-inverse) functor
\[
\mathcal D \colon  \Bds  \longrightarrow \Bds.
\]
%It will allow us to get for free a counterpart of Theorem \ref{Joyal-Tierney basic} for basic closed maps.
%Before presenting the result for basic closed maps, we introduce the following notation.
Now, given a b0-space $\mathcal{S}=(X, \mathcal{C}, \mathcal{B}d)$, set
\[\mathcal{C}^\mathfrak{f}(\mathcal{S}) =\rotatebox[origin=c]{180}{$\Lambda$}^\mathfrak{f}\Lambda^\mathfrak{f}\mathcal{C}(\mathcal{S}).\] Its elements are called {\em finitely generated closed domains}.
Further, given a basic continuous pair $(f,\lf)$ between b0-spaces $\mathcal{S}$ and $\mathcal{T}$, we will refer to the map
\[
		\lf^{\,\mathfrak{f} \cf} =( V\mapsto \lf[V]) \colon \mathcal{C}^\mathfrak{f}(\mathcal{T})  \to \mathcal{C}^\mathfrak{f}(\mathcal{S}),		
\]
as the {\em finite closed restriction of $\lf$}. (This map is well defined since the basic preimage of a closed domain is still a closed domain and $\lf$ preserves finite joins and finite meets.)

%Applying the preceding results to Theorem \ref{Joyal-Tierney basic}, we get immediately the following:

\begin{theo}\label{closedJT}
Let $\ms$ be a b0-space such that $\Lambda^\mathfrak{f}\mcs=\mcs$  and let $\mt$ be an arbitrary b0-space. Let  $(f,\lf)$ be a basic continuous pair between $\mathcal{S}$ and $\mathcal{T}$. Then $f$ is a basic closed map if and only if $\lf^{\,\mathfrak{f} \cf}$ has a left adjoint $h\colon \mathcal{C}^\mathfrak{f}(\mathcal{S})\to \mathcal{C}^\mathfrak{f}(\mathcal{T})$ with the following properties:
	\begin{enumerate}[\em (1)]
		\item For every $C\in \mathcal{C}(\mathcal{S})$, $h[C]\in \mathcal{C}(\mathcal{T})$.
		\item  For every $C\in \mathcal{C}(\mathcal{S})$ and every $D_1, \ldots , D_n\in \mathcal{C}(\mathcal{T})$
		\begin{equation*}
			h[C\cap \lf^{\,\mathfrak{f} \cf}[D_1 \cap\dots \cap D_n]]=h[C]\cap D_1 \cap \dots \cap D_n.
		\end{equation*}
	\end{enumerate}
\end{theo}

\begin{proof}
	Let $f\colon \mathcal{S}\to \mathcal{T}$ be a basic continuous map. It is clear that $f$ is a basic closed map if and only if $\widetilde{f}$ is a basic open map. Since $\Lambda^\mathfrak{f}\mo(\widetilde{\ms})=\Lambda^\mathfrak{f}\mcs=\mcs=\mo(\widetilde{\ms})$, we can apply  Theorem \ref{Joyal-Tierney basic} to conclude that $\widetilde{f}$ is basic open if and only if $\widetilde{f}_\leftarrow^{\, \mathfrak{f}\of}$ has a left adjoint $h\colon \mathcal{O}^\mathfrak{f}(\widetilde{\mathcal{S}})\to \mathcal{O}^\mathfrak{f}(\widetilde{\mathcal{T}})$ satisfying the following conditions:
	\begin{enumerate}[(1)]
		\item  For every $U\in \mathcal{O}(\mathcal{\widetilde{S}})$,\ $h[U]\in \mathcal{O}(\widetilde{\mathcal{T}})$.
		\item    For every $U\in \mathcal{O}(\widetilde{\mathcal{S}})$ and every $V_1, \ldots , V_n \in \mathcal{O}(\widetilde{\mathcal{T}})$
		\begin{equation*}
			h[U\cap \widetilde{f}_\leftarrow^{\, \mathfrak{f}\of}[V_1 \cap \dots\cap V_n]]=h[U]\cap V_1 \cap \dots \cap V_n .
		\end{equation*}
	\end{enumerate}
	Then the result follows immediately from the facts that $\widetilde{f}_\leftarrow^{\, \mathfrak{f}\of}=\lf^{\,\mathfrak{f} \cf}$, $\mathcal{O}(\widetilde{\mathcal{S}})=\mathcal{C}(\mathcal{S})$ and $\mathcal{O}(\widetilde{\mathcal{T}})=\mathcal{C}(\mathcal{T})$.
\end{proof}

This proof illustrates how the isomorphism $\mathcal{D}$  can be used to ``dualize'' results in the category of b0-spaces. It is now obvious that this category  has the following ``duality principle'':

\begin{prop}[Duality principle]
	Let $\Sigma$ be a statement in the category of basic zero-dimensional spaces. If $\Sigma$ is true, then its ``dual statement'' $\mathcal D(\Sigma)$ is also true. This dual statement is obtained by interpreting $\Sigma$ under $\mathcal D$ \textup(swapping ``closed'' and ``open'', $\lf^{\,\mathfrak{f} \cf}$ and $\lf^{\,\mathfrak{f} \of}$, etc.\textup).
\end{prop}

\begin{note}\label{rema dual}
	Observe that if we had required a basic continuous map $f\colon \mathcal{S}\to \mathcal{T}$ to satisfy $\pf[C]\in \mathcal{C}(\mathcal{S})$ for every $C\in \mathcal{C}(\mathcal{T})$, we would not get such a duality principle. In fact, it is possible for a basic continuous map $f$ to satisfy  $\pf[C]\in \mathcal{C}(\mathcal{S})$ for every $C\in \mathcal{C}(\mathcal{T})$, with $\pf[V]\not \in \mathcal{O}(S)$ for some $V\in \mathcal{O}(\mathcal{T})$.

This is what happens, in general, in the category of locales and localic maps that prevents us from having such a duality principle in point-free topology and, in particular, from obtaining the well-known characterization of closed maps as a corollary of the Joyal-Tierney Theorem.
\end{note}

\section{An application: Closed maps between Heyting lattices}

We finish with the generalization to Heyting lattices of a well-known characterisation of closed localic maps (\cite[Prop. III.7.3]{PP}), as an application of Theorem \ref{closedJT}.

\begin{theo}\label{closedtheo}
	The following are equivalent for a continuous map $f\colon H_1\to H_2$  between Heyting lattices:
	\begin{enumerate}[\em(i)]
		\item  $f$ is a closed map.
		\item $f$ is a basic closed map.
		\item  The basic $\cf$-preimage $f^\cf_{\leftarrow}$ has a left adjoint $h\colon \cf H_1\to \cf H_2$ such that
		$$
		h[\cf a\cap \lf[\cf b]]=h[\cf a]\cap \cf b \quad
		\mbox{for every }a\in H_1\mbox{ and }b\in H_2.$$
		\item For every $a\in H_1$ and $b\in H_2$, $f(a\vee f^\ast (b))=f(a)\vee b$.
		\item For every $a\in H_1$, $f[\cf a]= \cf (f(a))$.
		
	\end{enumerate}
\end{theo}

\begin{proof}
	(i)$\Leftrightarrow$(ii): It follows immediately from Remark \ref{rema basic closed=closed}.
	
	\nid
	(ii)$\Leftrightarrow$(iii): It follows immediately from Theorem \ref{closedJT} and the fact that, in any Heyting lattice $H$,  $\rotatebox[origin=c]{180}{$\Lambda$}^\mathfrak{f}\Lambda^\mathfrak{f}\cf H= \cf H$.
	
	\noindent (iii)$\Leftrightarrow$(iv): Since $f$ is a continuous map, $\lf^\cf(\cf b)=\cf (f^\ast (b))$. Moreover,  since
	\[
			\cf =( a \mapsto \cf a) \colon H  \to \cf H
	\]
	is an anti-isomorphism for every Heyting semilattice, there exists a one-to-one correspondence between maps $h\colon \cf H_1\to H_2$ and maps $f_!\colon H_1\to H_2$, such that for every $h$ its associated $f_!$ makes the following diagram to commute:
\[
\xymatrix@R=30pt@C=40pt{\cf H_1  \ar[r]^{h} & \cf  H_2 \ar[r]^(.55){\lf^\cf} & \cf H_1   \\
H_1 \ar[u]^{\cf} \ar[r]^{f_!} & H_2 \ar[u]^{\cf} \ar[r]^(.5){f^*} &  H_1. \ar[u]_{\cf}}
\]
	Now we use the fact that $\cf\colon H\to \cf H$ is an anti-isomorphism to conclude that $\lf^\cf$ has a left adjoint $h\colon\cf H_1\to \cf H_2$ if and only if $f^\ast\colon H_2\to H_1$ has a right adjoint $f_!\colon H_1 \to H_2$:
\[
\xymatrix@R=30pt@C=30pt{h    & \dashv & \lf^\cf    \\
f^\ast \ar[urr]^(.3){\cf}   & \dashv     &  f_!. \ar[ull]_(.3){\cf}}
\]
	But  $f$ is a continuous map; thus $f^\ast$ always has a right adjoint, namely $f$. It follows that $f_!=f$, and that $\clf$ always has a left adjoint, namely the map given by $h[\cf a]=\cf (f(a))$.
	
	Moreover, by the commutativity of the diagrams and the fact that $f_!=f$, we have
	$$\cf (f(a\vee f^\ast (b)))=h[\cf a \cap \clf[\cf b]] \qtq{and} \cf(f(a)\vee b)=h[\cf a]\cap \cf b. $$
	Hence
	$f(a\vee f^\ast (b))=f(a)\vee b\Leftrightarrow h[\cf a \cap \clf[\cf b]]=h[\cf a]\cap \cf b.$
	
	\noindent (iv)$\Rightarrow$(v): Since $f$ is continuous, it is monotone, and therefore $ f[\cf a]\sue \cf (f(a))$. On the other hand, let $b\in \cf (f(a))$. Then
	$b=f(a)\vee b= f(a\vee f^\ast (b))\in f[\cf a].$
	
	\noindent (v)$\Rightarrow$(i): It is trivial.
\end{proof}

\begin{rema}
In conclusion, we note that theorems \ref{Joyal-Tierney basic} and \ref{closedJT}, applied to the particular case of the category of locales, explain
	why open localic maps and closed localic maps behave so differently, although basic open and basic closed maps between b0-spaces are somehow dual to each other.

Note also that the only difference between the proofs of theorems \ref{JTtheo} and \ref{closedtheo} is that, in the former, $\of \colon L\to \of L$ is an isomorphism, whereas in the latter, $\cf \colon L\to \cf L$ is an anti-isomorphism. This implies that,
when $f$ is open, $f^\ast$ also has a {\em left} adjoint (arising from the left adjoint of $\lf^\of$). In contrast, when $f$ is closed, it implies that $f^\ast$ has a {\em right} adjoint, which is, after all, always true for arbitrary localic maps.
\end{rema}

\section*{Acknowledgments}

The second named author acknowledges partial financial support by the Centre for Mathematics of the University of Coimbra ({\tt doi.org/10.54499/ UID/00324/2025}), funded by the Portuguese Foundation for Science and Technology (FCT), grants UID/00324/2025 and UID/PRR/00324/2025.

\end{document}